\documentclass[11pt,epsf]{article}
\usepackage{graphicx}
\usepackage{amsthm}
\usepackage{amsfonts}
\usepackage{amssymb}
\title{Asymptotic expansions of zeros of a partial theta function}
\author{Vladimir Petrov Kostov\\ 
Universit\'e de Nice, 
Laboratoire de Math\'ematiques, Parc Valrose,\\ 06108 Nice Cedex 2, France,  
e-mail: vladimir.kostov@unice.fr} 
\date{}
\newtheorem{tm}{Theorem}

\newtheorem{rem}[tm]{Remark}
\newtheorem{rems}[tm]{Remarks}

\newtheorem{prop}[tm]{Proposition}

\begin{document} 
\maketitle 
\begin{abstract}
The bivariate series $\theta (q,x):=\sum _{j=0}^{\infty}q^{j(j+1)/2}x^j$ defines  
a {\em partial theta function}. For fixed $q$ ($|q|<1$), $\theta (q,.)$ is an 
entire function. We prove a property of stabilization of the 
coefficients of the Laurent series in $q$ of the zeros of 
$\theta$. The coefficients $r_k$ of the stabilized series 
are positive integers. 
They are the elements of a 
known increasing sequence satisfying the recurrence relation 
$r_k=\sum _{\nu =1}^{\infty}(-1)^{\nu -1}(2\nu +1)r_{k-\nu (\nu +1)/2}$. 


\end{abstract}

\section{Introduction}

The bivariate series $\theta (q,x):=\sum _{j=0}^{\infty}q^{j(j+1)/2}x^j$ 
(where $(q,x)\in \mathbb{C}^2$, $|q|<1$) defines  
a {\em partial theta function}. For fixed $q$, $\theta (q,.)$ is an 
entire function.

Different domains in which the partial theta function finds applications are 
asymptotic analysis (see \cite{BeKi}), the Ramanujan type $q$-series 
(see \cite{Wa}), the theory 
of (mock) modular forms (see \cite{BrFoRh}), statistical physics 
and combinatorics (see \cite{So}) and also some questions concerning hyperbolic 
polynomials (i.e. real polynomials with all roots real, see \cite{KaLoVi}, 
\cite{KoSh} and \cite{Ko2}). The latter are connected to a problem considered 
by Hardy, Petrovitch and Hutchinson, see \cite{Ha}, \cite{Hu}, \cite{Ost} 
and \cite{Pe}. Other properties of $\theta$ are considered in \cite{AnBe}.

In the article \cite{Ko4} the zeros of $\theta$ are presented in the form 
$-\xi _j=-1/q^j\Delta _j$. It is shown in \cite{Ko4} that $\Delta _j$ are 
{\em formal power series (FPS)} of the form $1+O(q)$. 
For $|q|\leq 0.108$ all zeros $-\xi _j$ 
are distinct (see \cite{Ko4}), so they depend analytically on $q$ and 
the series $\Delta _j$ converge. In the present paper we present the zeros 
of $\theta$ also as Laurent series of the form 
$-q^{-j}+a_jq^{\kappa _j}+o(q^{\kappa _j})$, where $\mathbb{Z}\ni \kappa _j>-j$.

\begin{tm}\label{tm1}
(1) The series $\Delta _j$ is of the form $1+(-1)^jq^{j(j+1)/2}\Phi _j(q)$, where 
$\Phi _j:=1+O(q)$
is an FPS in $q$. Hence $a_j=(-1)^j$ and 
$\kappa _j=j(j-1)/2$. 

(2) Represent the zeros $-\xi _j$ in the form 
$-\xi _j=-q^{-j}+(-1)^jq^{j(j-1)/2}(1+\sum _{k=1}^{\infty}g_{j,k}q^k)$. 
There exists an FPS of the form $(H)~:~1+\sum _{k=1}^{\infty}r_kq^k$, 
$r_k\in \mathbb{Z}$, such that $g_{j,k}=r_k$ for $k=1$, $\ldots$, $j$ and 
$j\geq 2$. 
The coefficients $r_k$ satisfy the following recurrency relation (we set 
$r_k:=0$ for $k<0$ and $r_0:=1$):

\begin{equation}\label{recur}
r_k=\sum _{\nu =1}^{\infty}(-1)^{\nu -1}(2\nu +1)r_{k-\nu (\nu +1)/2}~.
\end{equation} 
\end{tm}

\begin{rems}\label{remlist}
{\rm (1) The relation (\ref{recur}) implies that all numbers $r_k$ are integer. 
The first 20 of them are equal to} 

$$3~,~9~,~22~,~51~,~108~,~221~,~429~,~810~,~1479~,~2640~,~4599~,~7868~,~$$
$$13209~,~21843~,~35581~,~57222~,~90882~,~142769~,~221910~,~341649~.$$
{\rm The sequence $\{ r_k\}$ is known (see A000716 in Sloane's database of 
integer sequences \cite{Sl}). Set 
$(q)_{\infty}:=\prod _{k=1}^{\infty}(1-q^k)$. It is shown in \cite{GuZe} that  
$1+\sum _{k=1}^{\infty}r_kq^k=1/(q)_{\infty}^3$ and that, 
using Jacobi's triple product 
(see \cite{HW}, p. 377), one obtains the equality} 

\begin{equation}\label{qinfty3}
(q)_{\infty}^3=\sum _{j=0}^{\infty}(-1)^j(2j+1)q^{j(j+1)/2}~.
\end{equation}
{\rm With the help of this equality the authors of \cite{GuZe} obtain the 
recurrence formula (\ref{recur}).

(2) The functions $M(q):=(q)_{\infty}^3=1-3q+5q^3-\cdots$ and 
$(q)_{\infty}=1-q-q^2+\cdots$ are  
positive valued for 
$q\in (-1,1)$ and flat at $\pm 1$ 
(all this follows from $(q)_{\infty}:=\prod _{k=1}^{\infty}(1-q^k)$). 
They are decreasing on $[0,1)$ because every factor 
$1-q^k$ is positive valued and decreasing. Their derivatives at $0$ 
are negative, therefore they attain their maximal values (which are $>1$) 
at some point in $(-1,0)$ (the same for both of them). 
Their second derivatives at negative 
points close to $0$ are negative, so they have inflection points in $(-1,0)$. 
At $0$ the function $M$ has an inflection point. The function $(q)_{\infty}$ 
has an inflection point in $(0,1)$ (because it is flat at $1$ and 
$((q)_{\infty})''|_{q=0}<0$).} 

\end{rems}

The above remarks imply the following proposition:

\begin{prop}\label{prop1}
(1) The integer sequences $\{ r_k\}$ ($k\geq 0$), $\{ r_{k+1}-r_k\}$ 
($k\geq 0$) and  
$\{ r_{k+2}-2r_{k+1}+r_k\}$ ($k\geq 1$) are increasing. 

(2) The radius of convergence of 
the Taylor series $(H)$ of Theorem~\ref{tm1} 
equals $1$. 

\end{prop}

\begin{proof}[Proof of Proposition~\ref{prop1}]
Set $R(q):=\sum _{j=0}^{\infty}q^j$, $1/(q)_{\infty}^3:=R^3W$, i.~e. 
$W=(\prod _{k=2}^{\infty}R(q^k))^3$. 
The numbers $r_k$ are the coefficients of the series $1/(q)_{\infty}^3$. Denote 
by $s_k$, $t_k$ and $u_k$ the coefficients of the series 
$R^2W$, $RW$ and $W$. The equality $R=1+qR$ implies 

$$r_{k+1}=r_k+s_{k+1}~~,~~s_{k+1}=s_k+t_{k+1}~~{\rm and}~~t_{k+1}=t_k+u_{k+1}~.$$
As $r_k>0$, $s_k>0$, $t_k>0$ and $u_k>0$ for $k\geq 2$, 
this proves statement (1).
The radius of convergence of the right-hand side of (\ref{qinfty3}) is $1$ (and 
$(q)_{\infty}^3>0$ for $q\in [0,1)$), 
therefore this is the case of the series $(H)$ as well. This 
is statement~(2).

\end{proof}

\begin{prop}\label{M}
The function $M$ is convex on $(0,1)$.
\end{prop}

The proposition is proved in Section~\ref{secprM}.

\begin{rem}\label{remlist1}
{\rm The first 10 coefficients of the FPS 
$\Delta _1$, $\Delta _2$, $\Delta _3$ are listed below:}
$$\begin{array}{rrrrrrrrrr}
1&-1&-1&-1&-2&-4&-10&-25&-66&-178\\ 1&0&0&1&3&9&24&66&180&498\\ 
1&0&0&0&0&0&-1&-3&-9&-22\\ 
\end{array}$$
{\rm The reason why part (2) of the theorem does not hold true for $j=1$ 
is explained in Remark~\ref{remexplain}. It would be interesting to (dis)prove 
that for $k\geq 1$ the sequence $\{ r_{k+1}/r_k\}$ is decreasing.}
\end{rem}

{\bf Acknowledgement.} The author is grateful to B. Shapiro, A. Sokal and 
A. Eremenko for the discussions of this text.

\section{Proof of Theorem~\protect\ref{tm1}}

Prove part (1) of the theorem. 
Consider the condition $\theta (q,x)=0$. If instead of $-\xi _j$ 
one substitutes just $-q^{-j}$ for $x$ in $\theta$, then the negative powers of 
$q$ cancel in $\theta (q,-q^{-j})$. 
Indeed, denote by $\lambda _s$ the degree of the $s$th monomial 
of the Laurent series $\theta (q,-q^{-j})$ (i.e. the degree of the monomial 
$q^{s(s+1)/2}x^s|_{x=-q^{-j}}$). Hence $\lambda _s=(s^2+s)/2-js$.  

\begin{rem}\label{remlambda}
{\rm The first $j$ degrees $\lambda _s$ decrease from $\lambda _0=0$ to 
$\lambda _{j-1}=-j(j-1)/2$. Starting from $\lambda _{j}$, the degrees 
increase. One has} 

$$\begin{array}{ll}
(A)~:~&\lambda _{\nu}=\lambda _{2j-1-\nu}~~{\rm for}~~\nu =0,\ldots ,j-1~.\\ \\ 

(B)~:~&\lambda _{s+1}-\lambda _s=s+1-j~~ 
{\rm (for}~~s\geq j-1~~{\rm this~gives}~~
0, 1, 3, 6, 10,\ldots {\rm )~.}\end{array}$$
\end{rem}
 
The expansion of $\theta (q,-q^{-j})$ contains the monomials 
$(-1)^{\nu}q^{\lambda _{\nu}}$ and 
$(-1)^{2j-1-\nu}q^{\lambda _{\nu}}$ which cancel. 
When one substitutes $-q^{-j}+a_jq^{\kappa _j}+o(q^{\kappa _j})$ for $x$ in 
$\theta (q,x)$, one gets 

$$\Psi _s:=q^{s(s+1)/2}x^s|_{x=-q^{-j}+a_jq^{\kappa _j}+o(q^{\kappa _j})}=
(-1)^sq^{\lambda _s}+(-1)^{s-1}sa_jq^{\mu _s}+o(q^{\mu _s})~~,$$
where $\mu _s=s(s+1)/2+(s-1)(-j)+\kappa _j$. 
The lowest value of $\mu _s$ is attained for and only for $s=j-1$ and $s=j$. 
(One has $\mu _{j-1}=\mu _j=-j(j-3)/2+\kappa _j$.) In the Taylor 
series expansion 
of $\Psi _{j-1}$ and $\Psi _j$  
the sum of the coefficients of the two monomials with this power of $q$ equals 

$$S:=(-1)^j(j-1)a_j+(-1)^{j-1}ja_j=(-1)^{j-1}a_j~.$$
The quantity $S_1:=Sq^{\mu _j}$ 
must cancel with other monomials or with just another monomial 
from the 
expansion of $\theta (q,-\xi _j)$. This must be just one monomial, and it 
has to be  
$q^{\lambda _{2j}}$. Indeed, any monomial $T$ in the expansion 
of $\Psi _s$ for $s\neq j-1$ and $s\neq j$ which is not 
$q^{\lambda _s}$, has a degree higher than $\mu _j$. 
The same holds true for the monomials of $\Psi _{j-1}$ and $\Psi _j$ which are 
different from $\pm q^{\lambda _j}$ and $(-1)^{s-1}sa_jq^{\mu _j}$, $s=j-1$ or $j$.  
On the other hand, the monomials $q^{\lambda _s}$ cancel for 
$s=0$, $\ldots$, $2j-1$. Hence $S_1$ can cancel with (a) monomial(s) of degree 
at least $\lambda _{2j}$. The only monomial of degree $\lambda _{2j}$ is 
$q^{\lambda _{2j}}$ (where $\lambda _{2j}=(4j^2+2j)/2-2j^2=j$), 
and $S_1$ must cancel with it. Indeed, otherwise no monomial 
$T$ cancels with it either, i.e. the quantity $\theta (q,\xi _j)$ 
is not identically equal to $0$ which is a contradiction.
Hence $(-1)^{j-1}a_j+1=0$, i.e. $a_j=(-1)^j$, and 

$$-j(j-3)/2+\kappa _j=j~~{\rm hence}~~\kappa _j=j(j-1)/2~.$$
Thus the zero $\xi _j=-q^{-j}/\Delta _j$ is of the form 

$$-q^{-j}+(-1)^jq^{j(j-1)/2}+o(q^{j(j-1)/2})=
(-q^{-j})(1-(-1)^jq^{j(j+1)/2}+o(q^{j(j+1)/2}))~$$
which means that $\Delta _j=1+(-1)^jq^{j(j+1)/2}+o(q^{j(j+1)/2})$.  
This proves part (1).

We prove part (2) for $j\geq 4$. For $j=2$ and $j=3$ its proof is contained 
in Remark~\ref{remexplain} below. 
Recall that $-\xi _j$ is represented in the form 
$-q^{-j}+(-1)^j(q^{\kappa _j}+g_{j,1}q^{\kappa _j+1}+g_{j,2}q^{\kappa _j+2}+
o(q^{\kappa _j+2}))$. When one computes the expansion of $\Psi _s$ 
in powers of $q$, one applies the formula of the Newton binomial to the 
Laurent series of $-\xi _j$: 

(i) The term containing the 
first power of $q$ obtained in the 
expansion of $\Psi _s$ is $(-1)^sq^{\lambda _s}$.

(ii) The next ones are of the form $(-1)^{s-1}sq^{\lambda _s+j+\kappa _j}$ and 
$(-1)^{s-1}sg_{j,\nu}q^{\lambda _s+j+\kappa _j+\nu}$, $\nu =1$, $\ldots$, 
$j+\kappa _j-1$. 

(iii) To compute the higher powers of $q$ one has to take into account the 
monomial $(-1)^{s-2}(s(s-1)/2)q^{\lambda _s+2j+2\kappa _j}$ and then other monomials 
in which participate two or more of the 
factors $q^{\kappa _j}$, $g_{j,1}q^{\kappa _j+1}$, $g_{j,2}q^{\kappa _j+2}$ etc.      

(iv) Consider the sum $\Psi _{j-1-l}+\Psi _{j+l}$, $l=0$, $1$, $\ldots$, $j-1$. 
Its terms mentioned in (i) cancel. Its terms mentioned in (ii) equal
$(-1)^{j+l-1}(2l+1)q^{\lambda _{j+l}+j+\kappa _j}$ and 
$(-1)^{j+l-1}(2l+1)g_{j,\nu}q^{\lambda _{j+l}+j+\kappa _j+\nu}$, $\nu =1$, $\ldots$, 
$j+\kappa _j-1$. 

(v) The lowest power of $q$ in the expansion of $\Psi _{2j+r}$, 
$r=0$, $1$, $\ldots$,  
is $q^{\lambda _{2j+r}}$, where $\lambda _{2j+r}=jr+j+r(r+1)/2\geq j$. 

The following matrix illustrates the case $j=4$. We present 
the quantity $-\xi _4$ in the form 

$$-\xi _4=-q^{-4}+aq^{6}+bq^{7}+cq^{8}+dq^{9}+hq^{10}+uq^{11}+\cdots $$ 
with $b=g_{j,1}$, $c=g_{j,2}$, $\ldots$. 
(In fact, we know that $a=1$, but for the moment we prefer to keep $a$ 
as unknown quantity.) The first column 
indicates the power of $q$. The next columns show the coefficients 
of the corresponding powers of $q$ in the expansions of $\Psi _s$ 
for $s=0$, $\ldots$, $9$. This means, in particular, that 

$$\Psi _3=-q^{-6}+3aq^4+3bq^5+3cq^6+3dq^7+3hq^8+3uq^9+\cdots ~.$$
The entries $\pm 1$ of the matrix are the coefficients of $(-1)^sq^{\lambda _s}$.

$$\begin{array}{r|cccc|cccccc}
9&&c&-2h&3u&-4u&5h&-6c&&&-1 \\ 
8&&b&-2d&3h&-4h&5d&-6b&&& \\ 
7&&a&-2c&3d&-4d&5c&-6a&&& \\ 
6&&&-2b&3c&-4c&5b&&&& \\ 
5&&&-2a&3b&-4b&5a&&&& \\ 
4&&&&3a&-4a&&&&1& \\ \hline  
1~{\rm to}~3&&&&&&&&&& \\ \hline  
0&1&&&&&&&-1&& \\ 
-1&&&&&&&&&& \\ 
-2&&&&&&&&&& \\ 
-3&&-1&&&&&1&&& \\ 
-4&&&&&&&&&& \\ 
-5&&&1&&&-1&&&& \\ 
-6&&&&-1&1&&&&& \\  
\hline  
&\Psi _0&\Psi _1&\Psi _2&\Psi _3&\Psi _4&\Psi _5&\Psi _6&\Psi _7&
\Psi _8&\Psi _9 
\end{array}$$
As $\Psi _8=q^4+o(q^9)$ and $\Psi _9=-q^9+o(q^9)$, one obtains the following 
system of equations from which one finds the quantities $a$, $\ldots$, $h$ 
(one writes the conditions that the sums of the coefficients of 
$q^s$ are $0$, $s=9$, $8$, $\ldots$, $4$; the coefficient of $q^5$, 
for instance, equals 
$-2a+3b-4b+5a=-b+3a$):

\begin{equation}\label{systemmm}
\begin{array}{rrrrrrrrrrrrrrrr}
-u&+3h&&-5c&&&-1&=&0&~~~~~~~~&-c&+3b&&&=&0\\ \\ 
&-h&+3d&&-5b&&&=&0&&&-b&+3a&&=&0\\ \\
&&-d&+3c&&-5a&&=&0&&&&-a&+1&=&0
\end{array}
\end{equation}
The system is triangular and one readily finds that 

$$(~a~,~b~,~c~,~d~,~h~,~u~)~=~(~1~,~3~,~9~,~22~,~51~,~107~)~.$$
Now we describe what the analogs of the above matrix and the above system 
of equations are when $j$ is arbitrary. The analogs of the 
lines of the powers $4$ and 
$9$ of the matrix are the lines of the powers $j$ and $2j+1$. Their 
rightmost indicated entries equal $(-1)^{2j}=1$ and $(-1)^{2j+1}=-1$. 
The columns of $\Psi _{j-3}$, $\ldots$, $\Psi _{j+2}$ are given in the next 
matrix ($j$ is presumed greater than $4$):

$$\begin{array}{l|cccc|ccccc}
j+5&&(j-3)c&-(j-2)h&(j-1)u&-ju&(j+1)h&-(j+2)c&& \\ 
j+4&&(j-3)b&-(j-2)d&(j-1)h&-jh&(j+1)d&-(j+2)b&& \\ 
j+3&&(j-3)a&-(j-2)c&(j-1)d&-jd&(j+1)c&-(j+2)a&& \\ 
j+2&&&-(j-2)b&(j-1)c&-jc&(j+1)b&&& \\ 
j+1&&&-(j-2)a&(j-1)b&-jb&(j+1)a&&& \\ 
j&&&&(j-1)a&-ja&&&&1 \\ \hline  
1~{\rm to}~j-1&&&&&&&&& \\ \hline  
0&1&&&&&&&-1& \\ \hline 
\lambda _{j+2}&&(-1)^{j-3}&&&&&(-1)^{j+2}&& \\ 
\lambda _{j+2}-1&&&&&&&&& \\ 
\lambda _{j+1}&&&(-1)^{j-2}&&&(-1)^{j+1}&&& \\ 
\lambda _j&&&&(-1)^{j-1}&(-1)^j&&&& \\  
\hline  
&\Psi _0&\Psi _{j-3}&\Psi _{j-2}&\Psi _{j-1}&\Psi _j&\Psi _{j+1}&
\Psi _{j+2}&\Psi _{2j-1}&
\Psi _{2j} 
\end{array}$$
If one writes then the analog $(E)$ of system (\ref{systemmm}), 
only its first equation will be slightly different:
$-u+3h-5c=0$. Hence 

$$(~a~,~b~,~c~,~d~,~h~,~u~)~=~(~1~,~3~,~9~,~22~,~51~,~108~)~.$$
Denote by $(E_k)$ the equation of system $(E)$  
expressing the fact that the sum of the coefficients of $q^k$ must be $0$. 
(In system (\ref{systemmm}) we have written explicitly 
equations $(E_9)$ -- $(E_4)$.) 
Fix $j_0\in \mathbb{N}$. It is easy to see that for $j\geq j_0$ 
the equations $(E_j)$ -- $(E_{j+j_0})$ do not depend on $j$. Their 
form is (\ref{recur}) 
(this follows from statement (B) of Remark~\ref{remlambda}). 
Therefore the values of the first $j_0+1$ quantities $a$, $b$, $\ldots$,  
do not depend on $j$ for $j\geq j_0$. Part (2) is proved.

\begin{rem}\label{remexplain}
{\rm Part (2) of the theorem does not hold true for $j=1$ 
for the following reason. Consider the matrices of coefficients of the 
Laurent series $\Psi _j$. Recall that 
$\Psi _{s}=(-1)^s+(-1)^{s-1}sa_jq^{\mu _s}+\cdots$ with $\mu _s\geq j(j+1)/2$ 
for $s\geq 2j-1$. For 
$j>3$ the inequality $j(j+1)/2>2j$ holds true which means that the 
columns of $\Psi _{2j-1}$ and $\Psi _{2j}$  
(considered only for degrees of $q$ 
ranging from $-j(j-1)/2$ to $2j$) contain as only nonzero entries 
$\pm 1$ in the rows corresponding respectively to degree 
$0$ and $j$. (The columns of $\Psi _{\nu}$ for $\nu >2j$ contain only zeros 
in these rows.) Thus the system to which 
$(a,b,\ldots )$ is solution is completely defined (and 
according to one and the same rule) by the columns of 
the matrix containing the coefficients of $\Psi _0$, $\ldots$, 
$\Psi _{2j}$. 

For $j=2$ and $j=3$ one has to check directly that the system is 
of the same form, i. e.}

$$\begin{array}{lllllll}
-a+1=0&&-b+3a=0&&-c+3b=0&&{\rm and}\\ 
-a+1=0&&-b+3a=0&&-c+3b=0&&-d+3c-5a=0~.
\end{array}$$
{\rm For $j=1$ the system becomes $-a+1=0$, $-b+2a=0$ and the value of 
$b$ is not $3$, but $2$.} 
\end{rem}

\section{Proof of Proposition~\protect\ref{M}
\protect\label{secprM}}

Set $M=e^{3L}$, where $3L:=\ln M=3\sum _{k=1}^{\infty}\ln (1-q^k)$. Hence 
$M''=(9(L')^2+3L'')e^{3L}$ and it suffices to show that $3(L')^2+L''>0$ on 
$(0,1)$. Using Taylor series at $0$ of the logarithm one gets
$\displaystyle{L=-\sum _{j=1}^{\infty}(1/j)\sum _{k=1}^{\infty}q^{jk}=
-\sum _{j=1}^{\infty}q^j/j(1-q^j)=
-\sum _{j=1}^{\infty}(-1/j+1/j(1-q^j))}$, so 

$$\begin{array}{lccl}L'=-\sum _{j=1}^{\infty}q^{j-1}/(1-q^j)^2
&{\rm and}&L''=U+2V&{\rm ,~~where}\\ \\ 
U:=-\sum _{j=2}^{\infty}(j-1)q^{j-2}/(1-q^j)^2&,&V:=
-\sum _{j=1}^{\infty}jq^{2j-2}/(1-q^j)^3~.&\end{array}$$
In the expansion of $3(L')^2$ all terms are of the form 
$q^s/(1-q^k)^2(1-q^l)^2$. For $s$ fixed denote by $q^sS_s$ the sum of all 
these terms, i.e. $3(L')^2=\sum _{s=0}^{\infty}q^sS_s$. One has 

$$S_s=\left\{ \begin{array}{ll}
6\sum _{\nu =1}^{s/2}1/(1-q^{\nu})^2(1-q^{s+2-\nu})^2+3/(1-q^{s/2+1})^4&
{\rm for~}s~{\rm even}\\ \\
6\sum _{\nu =1}^{(s+1)/2}1/(1-q^{\nu})^2(1-q^{s+2-\nu})^2&
{\rm for~}s~{\rm odd.}\end{array}\right.$$
On the other hand, $L''=-\sum _{s=0}^{\infty}q^sT_s$, where 
$T_s:=(s+1)/(1-q^{s+2})^2+2(s+1)q^s/(1-q^{s+1})^3$. To prove that $3(L')^2+L''>0$ 
it suffices to show that for $s=0$, $1$, $\ldots$ one has $S_s>T_s$ for 
$q\in (0,1)$. 

Observe that 
$T_s<3(s+1)/(1-q^{s+1})^3$. One has 
$1/(1-q^{\nu})^2(1-q^{s+2-\nu})^2>1/(1-q^{s+1})^3$, 
$1/(1-q^{s/2+1})^4>1/(1-q^{s+1})^3$ and 
$1/(1-q^{\nu})^2(1-q^{s+2-\nu})^2>1/(1-q^{s+1})^3$ 
for the values of the indices $s$ and $\nu$ as indicated above. Hence 
$S_s>3(s+1)/(1-q^{s+1})^3>T_s$ from which the proposition follows.

\end{document}